\documentclass[12pt,letterpaper]{amsart}
\usepackage{amsmath,amssymb,amsfonts}
 \usepackage{amsaddr}
\usepackage[all]{xy}
\usepackage{graphicx}
\usepackage{caption}
\usepackage{enumerate}
\usepackage{xcolor}

\usepackage{appendix}
\usepackage[hidelinks]{hyperref}
\usepackage{comment}
\usepackage{bm}
\usepackage[a4paper, margin={2.5cm}]{geometry}
\usepackage{cancel}
\usepackage{xfrac}
\usepackage{cite}
\usepackage[normalem]{ulem}

\newtheorem{prop}{Proposition}
\newtheorem{lem}{Lemma}

\theoremstyle{definition}
\newtheorem{definition}{Definition}
\newtheorem{expl}{Example}

\newtheorem{rem}{Remark}

\DeclareMathOperator{\R}{\mathbb{R}}

\DeclareMathOperator{\Z}{\mathbb{Z}}

\numberwithin{equation}{section}

\renewcommand{\i}{\mathbf{i}}
\newcommand{\re}{{\text{re}}}
\newcommand{\im}{{\text{im}}}
\DeclareMathOperator{\conc}{\mathbin{\ast}}

\title{Equilibria in Kuramoto oscillator networks:\\ An algebraic approach}

\author{Tung T. Nguyen$^{1,2,3}$, Roberto C. Budzinski$^{1,2,3}$, Jacqueline Đo\`{a}n$^{1,2,3}$, Federico W.~Pasini$^{1,2,3}$, J\'an Min\'a{\v c}$^{1,3}$, Lyle E.~Muller$^{1,2,3}$}
\address{1 - Department of Mathematics, Western University, London, ON, Canada \\
2 - Brain and Mind Institute, Western University, London, ON, Canada \\
3 - Western Academy for Advanced Research, Western University, London, ON, Canada}

\begin{document}

\maketitle

\begin{abstract}
Kuramoto networks constitute a paradigmatic model for the investigation of collective behavior in networked systems. Despite many advances in recent years, many open questions remain on the solutions for systems composed of coupled Kuramoto oscillators on complex networks. In this article, we describe an algebraic method to find equilibria in this kind of system without approximation. To do this, we use a recently introduced algebraic approach to the Kuramoto dynamics, which results in an explicitly solvable complex-valued equation that captures the dynamics of the original Kuramoto model. Using this new approach, we obtain equilibria for both the nonlinear original Kuramoto and complex-valued systems considering the case of homogeneous natural frequency. We then completely classify all equilibria in the case of complete graphs originally studied by Kuramoto. Finally, we study equilibria in networks of coupled oscillators with phase-lag, in generalized circulant networks, multilayer networks, and also random networks.
\end{abstract}

\maketitle

\section{Introduction} 

A paradigmatic system for understanding the collective behavior of coupled nodes is given by the Kuramoto model, which was introduced by Yoshiki Kuramoto in 1975. This model can be described as a set of coupled phase-oscillators that interact through a nonlinear function, usually considered to be a sine function \cite{Kuramoto2012chemical, acebron2005kuramoto}. Studies of this mathematical model have revealed several, previously unknown and non-trivial phenomena that arise from the dynamics and pattern of connections in this system, such as phase synchronization, remote synchronization, cluster synchronization, chimera states, and Bellerophon states \cite{abrams2004chimera, arenas2008synchronization, boccaletti_2002, parastesh2020chimeras, rodrigues2016kuramoto, strogatz2000kuramoto, xu2018origin}. In this sense, this model has been used as a theoretical approach for studying collective behavior and emergent phenomena in different fields, spanning from social interactions to biology and physics \cite{arenas2008synchronization,boccaletti_2002, Strogatz2001, bick2020understanding}.

Many recent mathematical studies have focused on possible solutions of the Kuramoto model -- specifically, the equilibrium points of the system \cite{chen2018counting,xin2016analytical,chen2019three,mehta2015algebraic,coss2018locating}. One of the possible ways to mathematically analyze coupled Kuramoto oscillators is by considering a mean-field approach in the thermodynamic limit, i.e.~$N \rightarrow \infty$, where $N$ is the number of oscillators in the system \cite{pikovsky2015dynamics, pikovsky2008partially}. In this case, it is possible to define the critical coupling strength where the system transitions to synchronization \cite{basnarkov2007phase, pazo2005thermodynamic}. Using this approach, Medvedev has shown the existence of equilibrium points (twisted states) for small-world networks, where non-local connections play an important role \cite{medvedev2014small}. Also, the mean-field approximation is an important strategy in the analytical treatment of coupled oscillators, where the dynamics of the entire system is reduced to a few variables \cite{bick2020understanding, hu2014exact}. An interesting approach is given by the discrete state mean-field
Kuramoto Model \cite{li2022mean}. Furthermore, many other approaches have been used in order to explore the role of topology in the synchronization of Kuramoto oscillators. Recent advances have shown different synchronization phenomena in complex networks
\cite{parastesh2020chimeras, pikovsky2015dynamics, rodrigues2016kuramoto}.

An important question in this context is the possibility of finding equilibrium points given a specific network of Kuramoto oscillators \cite{lu2020synchronization,taylor2012there,townsend2020dense,yoneda2021lower}. Many authors have explored the role of the connection architecture in the existence and stability of equilibrium points, which affects the process of transition to a global phase-synchronized state. Interestingly, Townsend \textit{et al.} have recently shown, somewhat counter-intuitively, some examples where the stability of the synchronized state depends most strongly not on the density of connections in the network, but on the specific pattern of connections, highlighting the importance of a particular network's structure on Kuramoto dynamics \cite{townsend2020dense}. Even though these results represent a great advance in the understanding of the collective behavior, they reach an inherent difficulty:~the Kuramoto model is originally described by a nonlinear equation, which imposes a limitation to the analytical treatment, where methods of linearization and approximations are usually necessary.   
In recent work \cite{muller2021algebraic}, a novel analytical approach to Kuramoto oscillators has been introduced -- the complex-valued model -- which results in an explicitly solvable complex-valued equation that captures the dynamics of the original, nonlinear Kuramoto model. Specifically, using the explicit expression for the complex-valued model and an iterative, operator-based approach, we showed in \cite{budzinski2022geometry} that (1) the trajectories of the two systems (the original, nonlinear Kuramoto model and the complex-valued system) can match precisely for long times and (2) this approach can thus provide a unified, geometrical insight into the transient behavior of the networks. While the original Kuramoto model is defined in terms of real numbers, the complex-valued one is defined in terms of complex numbers, where the argument corresponds to solution of the original Kuramoto model.

In this work, we now use this analytical approach to study the equilibria in coupled Kuramoto oscillator networks. Because the complex-valued model has an explicit solution, we can approach this problem using very straightforward techniques, where no approximations are needed. We show the equilibria of the complex-valued model correspond to those of the original, real-valued, and nonlinear KM, using algebraic techniques on the adjacency matrix of an individual network model. The main ingredient of our analysis is the use of Euler's formula and some important properties of the matrix exponential. The results from this work provide insight into both the structure of equilibria in the original KM, in addition to providing new insight into the complex-valued analytical form we have introduced in previous work. 

This work is organized as follows. We first show that some eigenvectors of the network adjacency matrix are equilibrium points for the novel complex-valued model (Sec. \ref{sec:eq_points_kuramoto}). We next show that the equilibrium points for the complex-valued model are also equilibrium points for the original (nonlinear) Kuramoto model. Utilizing this result, we extend our analysis to the case where the network can be understood as a circulant graph (Sec. \ref{sec:eq_points_circulant}), focusing also on the typical case of all-to-all coupling (complete graph) (Sec. \ref{subsec:eq_points_complete_graph}). With this result obtained, we then extend our analysis of equilibria to four new cases:~(i) equilibria in networks of Kuramoto oscillators with phase-lag, which leads to an interplay between attractiveness and repulsiveness in the coupling term (Sec. \ref{sec:eq_points_phase_lag}), (ii) equilibria in generalized circulant networks (Sec. \ref{sec:eq_points_generalized_circulant}), (iii) equilibria in multilayer networks (Sec. \ref{sec:eq_points_non_circulant}), and (iv) a new method to design equilibria in Erd\H{o}s-R\'enyi random graphs by applying specific changes to the adjacency matrix (Sec. \ref{sec:eq_points_random_networks}). Taken together, these results extend the analysis of equilibria in networks of Kuramoto oscillators to cases that could not be considered previously, demonstrating the utility of our complex-valued analytical approach. 

\section{Equilibrium points of Kuramoto models}\label{sec:eq_points_kuramoto}

The original Kuramoto model can be defined as the dynamical system governed by the equation
\begin{equation}\label{eq:KM}
\frac{d\theta_i}{dt} = \omega_{i} + \epsilon \sum_{j=1}^{N} a_{ij} \sin( \theta_j - \theta_i ),
\end{equation}
where $\theta_{i}$ is the phase of the $i^{\mathrm{th}}$ oscillator, $\omega_{i}$ is its natural frequency, $\epsilon$ is the coupling strength, $N$ is the number of oscillators in the network and $\bm{A}=(a_{ij})$ defines the adjacency matrix. We initially consider $a_{ij}=  0$ if $i$ and $j$ are unconnected and $1$ if connected, and later consider weighted adjacency matrices with real-valued entries. In this paper, we deal with the \textit{homogeneous} Kuramoto model, where $\omega_{i}= \omega$ for all $i \in [1,N]$. In this case, considering a suitable rotating frame of reference, without loss of generality we can assume $\omega = 0$.

Following \cite{muller2021algebraic, budzinski2022geometry}, starting with the original (homogeneous) Kuramoto model, we can extend this to a complex-valued model, obtaining a new dynamical system, governed by the following equation (throughout the paper, we denote $\i=\sqrt{-1}$):
\begin{equation}\label{eq:analytical_KM}
\frac{d\theta_i}{dt} = \epsilon \sum_{j=1}^{N} a_{ij} \big[ \sin( \theta_j - \theta_i ) - \i \cos( \theta_j - \theta_i ) \big]\,,
\end{equation}
Note this expression now implies $\theta_i \in \mathbb{C}$. As shown in \cite{muller2021algebraic}, the above equation is equivalent to 
\begin{equation}
\frac{d}{dt} \left( e^{\i \bm{\theta}} \right) = \epsilon \bm{A} e^{\i\bm{\theta}},
\end{equation}
where $\bm{A}$ is the adjacency matrix. By letting $\bm{x} = e^{\i\bm{\theta}}$, we have
\begin{equation}
{\frac{d\bm{x}}{dt}} = \epsilon \bm{A} \bm{x}.
\label{eq:analytical_solution}
\end{equation}
Let $\bm{\theta}=\bm{\theta}_{\re}+\i \bm{\theta}_{\im}$ be the decomposition of $\bm{\theta}$ into the real and imaginary parts. Then, we have 
\begin{equation} \label{eq:argument}
\bm{x}=e^{\i \bm{\theta}_{\re}- \bm{\theta}_{\im}}=e^{-\bm{\theta_{\im}}} e^{\i \bm{\theta}_{\re}}.
\end{equation}
By Eq. (\ref{eq:argument}), we can observe that $\bm{\theta_{\re}}$ is the argument of the solution $\bm{x}$. This naturally leads to the following definition.
\begin{definition} \label{def:equi_point}
We say that $\bm{\theta}_0 \in [-\pi, \pi]^N$ is an equilibrium point of the complex-valued model -- Eq. (\ref{eq:analytical_KM}) -- if for all time $t \geq 0$, $\bm{\theta}(t)=\bm{\theta}_0$. Equivalently 
\[ \arg(\bm{x(t)})=\arg(e^{\epsilon t \bm{A}} \bm{x}_0)=\arg(\bm{x}_0)=\bm{\theta}_0,  \]
where $\bm{x}_0=e^{\i \bm{\theta}_0}$.
\end{definition} 

\begin{rem} 
It is important to emphasize that the original Kuramoto model and our complex-valued approach are two distinct dynamical systems. We find that, by iterating the explicit expression for the complex-valued system over short intervals, the trajectories of the two systems can precisely match for long times \cite{budzinski2022geometry}. This approach thus offers analytical, mechanistic insights into the transient dynamics of Kuramoto networks, where the dynamics can be captured in terms of the eigenmodes of the system \cite{budzinski2022geometry}. In this work, we analyze the equilibrium points of both systems. At the equilibrium points considered here, the two models are equivalent. Since the systems have this equivalence in this context, we use the same variable $\theta$ to refer to both systems.
\end{rem}

We also recall the definition of equilibrium points in the theory of differential equations (see \cite{Perko2001} for further discussions).
\begin{definition} \label{def:classical_equi}
Consider the following differential equation
\begin{equation} \label{eq:classical_equi}
\dot{\bm{\theta}}(t)=f(\bm{\theta}(t)) ,
\end{equation}
where $\bm{\theta}(t)=(\theta_1(t), \theta_2(t), \ldots, \theta_N(t)).$ We say that $\bm{\theta}_0$ is an equilibrium point for Eq. (\ref{eq:classical_equi}) if $f(\bm{\theta}_0)=0.$ In other words, $\bm{\theta}(t)=\bm{\theta}_0$ is a solution of the following differential equation with initial condition 
\[
 \begin{cases} \dot{\bm{\theta}}(t)=f(\bm{\theta}(t)) \\
\bm{\theta}(0)=\bm{\theta}_0. \end{cases}
\] 
\end{definition}
In the context of the original Kuramoto model Eq. (\ref{eq:KM}) (note that we have set $\omega=0$), $\bm{\theta}_0=(\theta_1, \theta_2, \ldots, \theta_N)$ is an equilibrium point if and only if for all $1 \leq i \leq N$ 
\[ \sum_{j=1}^{N} a_{ij} \sin( \theta_j - \theta_i ) =0 .\]

\begin{rem}
In this article, we concentrate on the classical notion of equilibrium points as discussed in Definition \ref{def:classical_equi} and Definition \ref{def:equi_point}. In our subsequent work, we will discuss a more general notion of equilibria in Kuramoto networks, namely those solutions in which $\theta_j(t)-\theta_i(t)$ is a constant for all time $t$. 

\end{rem} 

In the following discussion, we develop an effective method to find equilibrium points of the original KM -- Eq. (\ref{eq:KM}) -- and the complex-valued model -- Eq. (\ref{eq:analytical_KM}).

We start with the following lemma.
\begin{lem} \label{lem:exponential}
Let $\bm{B}$ be a matrix. Suppose $\lambda$ is an eigenvalue of a matrix $\bm{B}$ and $\bm{v}$ is an associated eigenvector. Then 
\[ e^{\bm{B}}v=e^{\lambda} \bm{v} .\] 
\end{lem} 
\begin{proof}
By definition we have 
\[ e^{\bm{B}}=\sum_{n=0}^{\infty} \frac{\bm{B}^{n}}{n!} .\]
Applying both sides to $\bm{v}$ we have 
\[ e^{\bm{B}}\bm{v}=\sum_{n=0}^{\infty} \frac{\bm{B}^{n}}{n!}\bm{v}=\sum_{n=0}^{\infty} \frac{\lambda^n}{n!}\bm{v}=e^{\lambda}\bm{v} .\] 
\end{proof}

A direct consequence of the above lemma is the following.
\begin{prop} \label{prop:twisted_state}
Suppose $\bm{x}_0=e^{\i \bm{\theta}_0}$ is an eigenvector of $\bm{A}$ associated with a real eigenvalue $\lambda$. Then $\bm{\theta}_0$ is an equilibrium point of the complex KM in the sense of Definition \ref{def:equi_point}.
\end{prop} 
\begin{proof}
By Lemma \ref{lem:exponential} for $\bm{B}=\epsilon t \bm{A}$, we have 
\[ \bm{x}(t)=e^{\epsilon t \bm{A}} \bm{x}_0=e^{\lambda \epsilon t} \bm{x}_0 .\] 
Because $\lambda \in \R$ we have 
\[ \arg(\bm{x}(t))=\arg(e^{\lambda \epsilon t} \bm{x}_0)=\arg(\bm{x}_0)=\bm{\theta}_{0} .\] 
\end{proof}

Next, we show that the above equilibrium points are also equilibrium points for the original KM Eq. (\ref{eq:KM}).
\begin{prop} \label{prop:twisted_state_KM}
Suppose $\bm{x}_0=e^{\i \bm{\theta}_0}$ is an eigenvector of $\bm{A}=(a_{ij})$ associated with a real eigenvalue $\lambda$. Then $\bm{\theta}_0=(\theta_1, \theta_2, \ldots, \theta_N)$ is an equilibrium point of the original KM.
\[ \frac{d\theta_i}{dt} = \epsilon \sum_{j=1}^{N} a_{ij} \sin( \theta_j - \theta_i ). \] 
\end{prop} 

\begin{proof}
By Definition \ref{def:classical_equi}, we need to show that for all $1 \leq i \leq N$ 
\[ \sum_{j=1}^N a_{ij} \sin(\theta_j-\theta_i)=0 .\] 

First of all, since $\bm{x}_0$ is an eigenvector, we know that for all $1 \leq i \leq N$ 
\[ \sum_{j=1}^N a_{ij} e^{\i \theta_j}=\lambda e^{\i \theta_i} .\]

Taking the conjugation of both sides and noting that $a_{ij} \in \R$ and $\lambda \in \R$, we have 
 \[ \sum_{j=1}^N a_{ij} e^{-\i \theta_j}=\lambda e^{-\i \theta_i} .\]
We recall Euler's formula 
\begin{equation}\label{eq:euler} \sin(x)=\frac{e^{\i x}-e^{-\i x}}{2 \i},
\end{equation}
thanks to which
\[ \sin(\theta_j-\theta_i)=\frac{e^{\i (\theta_j-\theta_i)}-e^{-\i (\theta_j -\theta_i)}}{2 \i } .\] 
Hence 
\begin{align*}
(2 \i) \sum_{j=1}^Na_{ij} \sin(\theta_j-\theta_i) &= \sum_{j=1}^N a_{ij} \left[ e^{\i (\theta_j-\theta_i)}- e^{-\i(\theta_j-\theta_i)} \right] \\
&=e^{-\i \theta_i} \sum_{j=1}^N a_{ij} e^{\i \theta_j} - e^{\i \theta_i} \sum_{j=1}^N a_{ij} e^{-\i \theta_j}\\
&= e^{-\i \theta_i} \lambda e^{\i \theta_i}-e^{\i \theta_i} \lambda e^{-\i \theta_i}\\
&=\lambda-\lambda=0.
\end{align*} 
This completes the proof.
\end{proof}

\begin{rem}\label{rem:theta_1}
We state many results assuming that equilibrium points have a zero first coordinate, that is, the phase of the first oscillator is $0$. Since at equilibrium points all oscillators have a constant phase, this assumption is not a real loss of generality, but rather it is a normalization obtained as part of the choice of the rotating frame of reference. However, we remind the reader that adding a fixed phase to all coordinates of an equilibrium point gives another equilibrium point.
\end{rem}

\section{Equilibrium points for circulant networks} \label{sec:eq_points_circulant}

We apply Proposition \ref{prop:twisted_state} and Proposition \ref{prop:twisted_state_KM} to find equilibrium points of the the KM when the topological structure of the network is circulant. More precisely, let $\bm{C}$ be a symmetric circulant matrix with first column $[c_0, c_1, \ldots,c_{n-1}]$. Then the Circulant Diagonalization Theorem ensures that, defining
\begin{equation}
    \bm{\theta}^{(j)}=\left(0,  \frac{2 \pi j}{n}, \ldots, \frac{2 \pi j (n-1)}{n} \right)^{T}\qquad\mbox{for }j=0,\dots,n-1,
    \label{eq:twisted_state_eq}
\end{equation}
the vector $\bm{x}_0^{(j)}=e^{\i \bm{\theta}^{(j)}}$ is an eigenvector of $\bm{C}$ associated with the (real) eigenvalue 
\begin{equation*}
    \lambda _{j}=c_{0}+c_{n-1}\omega_n ^{j}+c_{n-2}\omega_n^{2j}+\dots +c_{1}\omega_n ^{(n-1)j}.
\end{equation*}
See also \cite{townsend2020dense}.

By Proposition \ref{prop:twisted_state} and Proposition \ref{prop:twisted_state_KM}, we have
\begin{prop}
For $0 \leq j \leq n-1$, let 
\[ \bm{\theta}^{(j)}=\left(0,  \frac{2 \pi j}{n}, \ldots, \frac{2 \pi j (n-1)}{n} \right)^{T} . \] 
Then 
\begin{enumerate}
\item $\bm{\theta}^{(j)}$ is an equilibrium point of the complex-valued -- Eq. (\ref{eq:analytical_solution}). 
\item $\bm{\theta}^{(j)}$ is an equilibrium point of the original KM -- Eq. (\ref{eq:KM}).
\end{enumerate}
\end{prop} 

Based on these results, Fig. \ref{fig:eq_point_circulant_phases} portrays a graphic representation of an equilibrium point for a circulant network of Kuramoto oscillators, which is also called as a ``twisted state". Here, a system with $N = 50$ oscillators connected with a ring network, where each node has $k = 10$ connections in both directions. This solution follows Eq. (\ref{eq:twisted_state_eq}) with $j = 1$ (a), which is equivalent to take the $2^{\mathrm{nd}}$ eigenvector of the matrix $\bm{A}$. This leads to an equilibrium point where the oscillators' phases are equally spaced between $0$ and $2\pi$ (represented in color-code). Moreover, we also consider the case where the equilibrium point follows Eq. (\ref{eq:twisted_state_eq}), but with $j = 3$ (b), which is equivalent to take the $4^{\mathrm{th}}$ eigenvector of $\bm{A}$. In this case, the distribution of the phases changes, but it is still an equilibrium point. These kind of solution is also known as ``twisted states".
\begin{figure}[htb]
    \centering
    \includegraphics[width=0.75\textwidth]{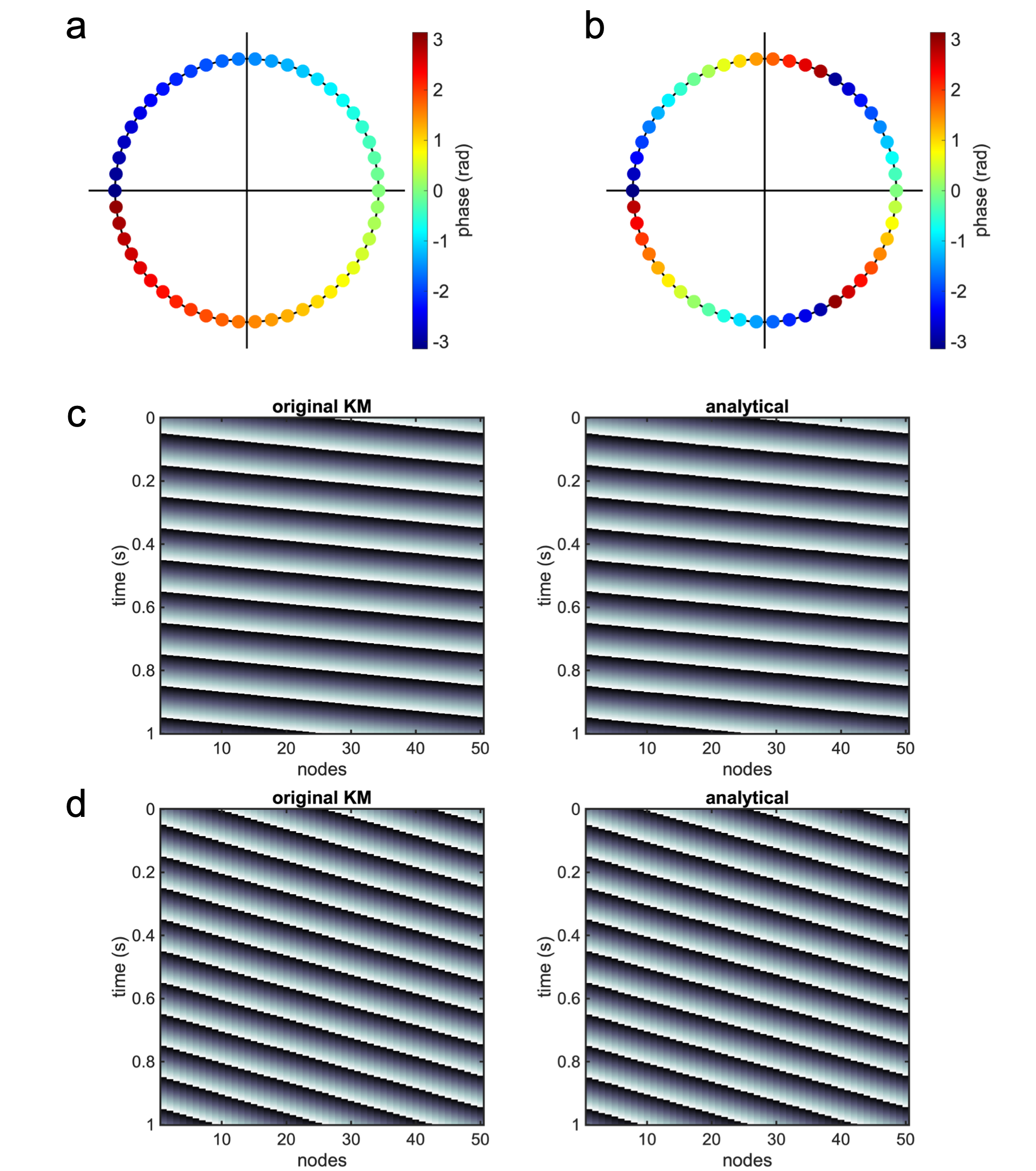}
    \caption{Graphic representation of an equilibrium point for a circulant network following Eq. (\ref{eq:twisted_state_eq}) with $j = 1$ (a) and $N = 50$ nodes. This is equivalent to take the $2^{\mathrm{nd}}$ eigenvector of the matrix $\bm{A}$. The phase of each oscillator is represented in color-code, which is equally distributed between $0$ and $2\pi$. We also consider the case with $j = 3$, which is equivalent to take the $4^{\mathrm{th}}$ eigenvector of the matrix $\bm{A}$ (b). This kind of equilibrium point is also known as a ``twisted state". Spatiotemporal dynamics for both the original Kuramoto model and the new complex-valued one (analytical) depict the same behavior. Here, the initial conditions for panel (c) are given by the equilibrium point represented in (a), while for panel (d) are given by (b). In both cases, the difference between the oscillators' phases remains as time evolves.}
    \label{fig:eq_point_circulant_phases}
\end{figure}

Then, we use the equilibrium points represented in Figs. \ref{fig:eq_point_circulant_phases}a and \ref{fig:eq_point_circulant_phases}b as initial conditions for the temporal evolution of the system. The spatiotemporal dynamics for both the numerical (original) Kuramoto model and the complex-valued version (analytical) is represented in (c) and (d), where the oscillators' phases are represented in color-code, and the initial conditions are given, respectively, by the solutions depicted in (a) and (b). We observe that the initial configuration remains as time evolves, which corroborates the solution as an equilibrium point for both cases. The oscillators' natural frequency is given by $\omega = 20 \pi$, and the coupling strength is given by $\epsilon = 1.0$. 

Furthermore, one can observe that a difference between the dynamics of these two cases, where the ``shape" of these waves differs as the equilibrium point changes. In this case, the wave pattern in (c) depicts diagonal structures with a lower slope (in comparison with the horizontal line) than (d) since the solutions are given by the $2^{\mathrm{nd}}$ and $4^{\mathrm{th}}$ eigenvectors, respectively.

\subsection{Complete classification of equilibrium points for complete graphs}\label{subsec:eq_points_complete_graph}

A particular case of circulant networks is given by complete networks, in which every oscillator is connected to any other. In other words, the topological connection is given by the complete graph $K_N$. In this case, based on the previous results, we can give an exhaustive characterization of equilibrium points.

Let $\bm{A}_N$ be the adjacency matrix of $K_N$. We have the following proposition.
\begin{prop}\label{prop:K_N}
Let $\bm{\theta}_0=(\theta_1, \ldots, \theta_N)$ be an initial condition such that 
\[ \sum_{k=1}^N e^{\i \theta_k}=0 .\] 
Then $e^{\i \bm{\theta}_0}$ is an eigenvector of $\bm{A}_N$ associated with the eigenvalue $\lambda=-1$. Consequently, $\bm{\theta}_0$ is an equilibrium point of the complex-valued model -- Eq. (\ref{eq:analytical_KM}) -- and the original KM -- Eq. (\ref{eq:KM}).
\end{prop} 
\begin{proof}
We have
\[ (\bm{A}_N+\bm{I}) \begin{pmatrix} e^{\i \theta_1} \\ \vdots \\ e^{\i \theta_N} \end{pmatrix}= 
 \begin{pmatrix}
\sum_{k=1}^N e^{\i \theta_k}\\
\vdots \\
\sum_{k=1}^N e^{\i \theta_k}
\end{pmatrix} = \begin{pmatrix}  0 \\ \vdots \\ 0 \end{pmatrix} .\] 

In other words, $e^{\i\bm{\theta}_{0}}$ is an eigenvector of $\bm{A}_N$ associated with the eigenvalue $\lambda=-1$. By Proposition \ref{prop:twisted_state} and Proposition \ref{prop:twisted_state_KM}, $\bm{\theta}_0$ is an equilibrium point of the complex-valued model -- Eq. (\ref{eq:analytical_KM}) -- and the original KM -- Eq. (\ref{eq:KM}) -- as claimed.
\end{proof}

We can go further to classify all equilibrium points of the classical Kuramoto model -- Eq. (\ref{eq:KM}). Suppose $\bm{\theta}_0=(\theta_1, \ldots, \theta_N)$ is an equilibrium point of the KM -- Eq. (\ref{eq:KM})--not encompassed by Proposition \ref{prop:K_N}, that is, satisfying 
\[ \sum_{i=1}^{N} e^{\i \theta_i} \neq 0 .\] 
By taking the conjugate of both sides, this also implies that 
\[ \sum_{i=1}^{N} e^{-\i \theta_i} \neq 0 .\] 

Since $\bm{\theta}_0=(\theta_1, \ldots, \theta_N)$ is an equilibrium point, for all $1 \leq i \leq N$ 
\begin{equation} \label{eq:K_N}
 \sum_{j=1}^N \sin(\theta_j-\theta_i)=0.
\end{equation} 
Using Euler's formula (\ref{eq:euler}), Eq. (\ref{eq:K_N}) becomes

\[ \sum_{j=1}^N e^{\i (\theta_j -\theta_i)} -\sum_{j=1}^N e^{-\i (\theta_j-\theta_i)} =0.\] 

We can rewrite this as 
\[ e^{-\i \theta_i} \left(\sum_{j=1}^N e^{\i \theta_j} \right) -e^{ \i \theta_i} \left( \sum_{j=1}^N e^{-\i \theta_j} \right) =0 .\] 

We then see that 
\[ e^{2 \i \theta_i}=\dfrac{\sum_{j=1}^N e^{\i \theta_j}}{\sum_{j=1}^N e^{-\i \theta_j}} .\] 

As the right hand side does not depend on $i$, we conclude that 
\[ e^{2 \i \theta_i}=e^{2 \i \theta_j} , \forall 1 \leq i ,j \leq N .\] 
Assuming, up to a rotation of the frame of reference, that $\theta_1=0$ (see Remark \ref{rem:theta_1}). Then the above equation implies that 
\[ e^{2 \i \theta_j}=1 ,\forall 1 \leq j \leq N .\] 
In other words $\theta_j \in \{0, \pi \}$. It is straightforward to check that if $\theta_j \in \{0, \pi \}$ then $\bm{\theta}_{0}$ is indeed an equilibrium point of Eq. (\ref{eq:KM}). Note that the property $\theta_j=0$ for all $j$ is a special case of this condition. In summary, we just proved the following. 
\begin{prop}
The point $\bm{\theta}_0=(\theta_1, \ldots, \theta_N)$ is an equilibrium point of the KM -- Eq. (\ref{eq:KM}) -- if and only if $\bm{\theta}_0$ satisfies one of the following conditions. 
\begin{enumerate}
\item The $\theta_i$s differ by integer multiples of $\pi$.
\item $\sum\limits_{i=1}^{N} e^{\i \theta_i} = 0$.
\end{enumerate} 
\end{prop} 

\section{Equilibrium points for phase-lag coupled oscillators}\label{sec:eq_points_phase_lag}

In this section, we generalize the results in the previous section to find equilibrium points for the Kuramoto networks where the oscillators are coupled with phase-lag. 

The original KM can be described, in this case, as:
\begin{equation} \label{eq:phase_lag}
    \frac{d\theta_{i}}{dt} = \omega_{i} + \epsilon \sum\limits_{j=1}^{N} a_{ij} \sin{(\theta_{j} - \theta_{i} - \phi)},
\end{equation}
where $N$ is the number of oscillators, $\epsilon$ is coupling strength, $\bm{A}=(a_{ij})$ defines the adjacency matrix, and $\phi$ is the phase-lag. Using this new parameter $\phi$ we can transition from a purely attractive coupling ($\phi = 0$) to a purely repulsive one ($\phi = \sfrac{\pi}{2}$). Here, following the same idea described in Sec. \ref{sec:eq_points_kuramoto}, we assume all intrinsic frequencies $\omega_i$ to be equal and we consider the rotating frame, in which $\omega_{i} = 0  \hspace{0.3cm} \left(\forall i \in [1,N]\right)$.

Furthermore, using the same technique of complexification, the complex-valued model can be described as:
\begin{equation} \label{eq:phase_lag_analytic}
\frac{d \theta_i}{dt}=\epsilon \sum_{j=1}^N a_{ij} \left(\sin(\theta_j-\theta_i-\phi)-\i \cos(\theta_j-\theta_i-\phi) \right).
\end{equation} 
By Euler's formula we have 
\begin{equation*}
\i\sin(\theta_j-\theta_i-\phi)+ \cos(\theta_j-\theta_i-\phi) =e^{\i(\theta_j-\theta_i-\phi)}.
\end{equation*} 
The complex-valued model becomes 
\begin{equation*}
\i \frac{d \theta_i}{dt}= \epsilon e^{-\i \theta_i} \sum_{j=1}^N a_{ij} e^{-\i \phi} e^{\i \theta_j}.
\end{equation*} 
Equivalently
\begin{equation*}
 \i e^{\i \theta_i}  \frac{d \theta_i}{dt}= \epsilon \sum_{j=1}^N a_{ij} e^{-\i \phi} e^{\i \theta_j}.
\end{equation*} 
Let $x_i=e^{\i \theta_i}$. Then $\frac{dx_{i}}{dt}= \i e^{\i \theta_i}  \frac{d \theta_i}{dt}$. Therefore the above equation becomes 
\begin{equation*} 
\frac{dx_{i}}{dt}= \epsilon \sum_{j=1}^N a_{ij} e^{-\i \phi} x_{j}.
\end{equation*}

The general solution of this linear ODE is
\begin{equation}
\bm{x}= e^{t \bm{K}} \bm{x}(0),
\label{eq:solution_analytical_model}
\end{equation}
where $\bm{x}=(x_1, \ldots, x_N)$ and $\bm{K}=\epsilon e^{-\i \phi} \bm{A}$. 

Finally, we explain how to get a real solution out of the complex solution for the complex-valued model described by Eq. (\ref{eq:phase_lag_analytic}). Let $\bm{\theta}=\bm{\theta}_{\re}+\i \bm{\theta}_{\im}$ be the decomposition of $\bm{\theta}$ into the real and imaginary parts. Then we have 
\begin{equation}
\bm{x}=e^{\i \bm{\theta}_{\re}- \bm{\theta}_{\im}}=e^{-\bm{\theta}_{\im}} e^{\i \bm{\theta}_{\re}}.
\end{equation}
We see that $\bm{\theta}_{\re}$ is thus the argument of the analytical solution $\bm{x}$. In particular, we can take $\bm{\theta}_{\re} \in [-\pi, \pi]$. 

We have the following definition which naturally generalizes Definition \ref{def:equi_point}.
\begin{definition} \label{def:equi_point_phase_lag}
We say that $\bm{\theta}_0 \in [-\pi, \pi]^N$ is an equilibrium point of the phase-lag complex-valued model -- Eq. (\ref{eq:phase_lag_analytic}) if for all times $t \geq 0$ 
\[ \arg(\bm{x(t)})=\arg(e^{\epsilon e^{-\i \phi} \bm{A} t} \bm{x}_0)=\arg(\bm{x}_0)=\bm{\theta}_0,  \]
where $\bm{x}_0=e^{\i \bm{\theta}_0}$.
\end{definition} 

We have the following proposition about equilibrium points of the complex-valued. 
\begin{prop} \label{prop:twisted_stated_phase_lag}
Suppose $\bm{x}_0=e^{\i \bm{\theta}_0}$ is an eigenvector of $\bm{A}=(a_{ij})$ associated with the eigenvalue $\lambda$. Suppose further that $\lambda e^{-\i \phi} \in \R$. Then $\bm{\theta}_0=(\theta_1, \theta_2, \ldots, \theta_N)$ is an equilibrium point of the phase-lag complex-valued model -- Eq. (\ref{eq:phase_lag_analytic})
 \[\frac{d \theta_i}{dt}=\epsilon \sum_{j=1}^N a_{ij} \left(\sin(\theta_j-\theta_i-\phi)-\i \cos(\theta_j-\theta_i-\phi) \right).\]
\end{prop} 
\begin{proof}

Because $\bm{x}_0$ is an eigenvector of $\bm{A}$ associated with $\lambda$, we have $\bm{A} \bm{x}_0= \lambda \bm{x}_0$. Therefore, applying Lemma \ref{lem:exponential} with $B=t\epsilon e^{-\i \phi}\bm{A}$, we have \[\bm{x}(t)=e^{t\epsilon e^{-\i \phi}\bm{A}} \bm{x}_0=e^{t\epsilon e^{-\i \phi}\lambda} \bm{x}_0. \]

Taking the argument of both sides we have 
\[ \arg(\bm{x}(t))=\arg(e^{t \epsilon e^{-\i \phi}\lambda} \bm{x}_0)=\text{Im}(t \epsilon e^{-\i \phi}\lambda) \arg(\bm{x}_0) =\arg(\bm{x}_0)=\bm{\theta}_0.\] 
Here we use the crucial assumption that $\lambda e^{-\i \phi} \in \R$. We conclude that $\bm{\theta}_0$ is an equilibrium point of the phase-lag complex-valued model in the sense of Definition \ref{def:equi_point_phase_lag}.
\end{proof}

Next, we show that the equilibrium points described in Proposition  \ref{prop:twisted_stated_phase_lag} are equilibrium points of the original phase-lag KM (\ref{eq:phase_lag}) as well (note that we have set $\omega=0$.) This is a direct generalization of Proposition \ref{prop:twisted_state_KM}. 
\begin{prop} \label{prop:twisted_stated_phase_lag_KM}
Suppose $\bm{x}_0=e^{\i \bm{\theta}_0}$ is an eigenvector of $\bm{A}=(a_{ij})$ associated with the eigenvalue $\lambda$. Suppose further that $\lambda e^{-\i \phi} \in \R$ (or equivalently $-\phi+\arg(\lambda)$ is an integer multiple of $\pi$). Then $\bm{\theta}_0=(\theta_1, \theta_2, \ldots, \theta_N)$ is an equilibrium point of the phase-lag KM -- Eq. (\ref{eq:phase_lag})
  \[\frac{d\theta_{i}}{dt} = \epsilon \sum\limits_{j=1}^{N} a_{ij} \sin{(\theta_{j} - \theta_{i} - \phi)}. \] 
\end{prop}

\begin{proof}
We need to show that for $1 \leq i \leq N$
\[ \sum\limits_{j=1}^{N} a_{ij} \sin{(\theta_{j} - \theta_{i} - \phi)} =0 .\]

First of all, since $\bm{x}_0$ is an eigenvector, we know that for all $1 \leq i \leq N$ 
\[ \sum_{j=1}^N a_{ij} e^{\i \theta_j}=\lambda e^{\i \theta_i} .\]

Taking the conjugation of both sides and noting that $a_{ij} \in \R$, we have 
 \[ \sum_{j=1}^N a_{ij} e^{-\i \theta_j}=\bar{\lambda} e^{-\i \theta_i} .\]
By Euler's formula \eqref{eq:euler}
\[ \sin(\theta_j-\theta_i-\phi)=\frac{e^{\i (\theta_j-\theta_i-\phi)}-e^{-\i (\theta_j -\theta_i-\phi)}}{2 \i } .\] 
Hence 
\begin{align*}
(2 \i) \sum_{j=1}^Na_{ij} \sin(\theta_j-\theta_i-\phi) &= \sum_{j=1}^N a_{ij} \left[ e^{\i (\theta_j-\theta_i-\phi)}- e^{-\i(\theta_j-\theta_i-\phi)} \right] \\
&=e^{-\i \theta_i}e^{-\i \phi} \sum_{j=1}^N a_{ij} e^{\i \theta_j} - e^{\i \theta_i}e^{\i \phi} \sum_{j=1}^N a_{ij} e^{-\i \theta_j}\\
&= e^{-\i \theta_i} e^{-\i \phi} \lambda e^{\i \theta_i}-e^{\i \theta_i} e^{\i \phi} \bar{\lambda} e^{-\i \theta_i}\\
&=\lambda e^{-\i \phi}-\bar{\lambda}e^{\i \phi} \\ &= \lambda e^{-\i \phi}-\overline{\lambda e^{-\i \phi}}=0.
\end{align*} 
The last inequality comes from the assumption that $\lambda e^{-\i \phi} \in \R$. This completes the proof.
\end{proof}
Note that by the same argument, we have the following slightly more general statement.   
\begin{prop}\label{prop:twisted_stated_phase_lag_KM_gen}
Let $\bm{x}_0=e^{\i \bm{\theta}_0}=(e^{\i \theta_1}, \ldots, e^{\i \theta_N})$. Suppose that for each $1 \leq i \leq N$, there exists a pair $(\lambda_i, \phi_i)$ such that $\lambda_i e^{-\i \phi_i} \in \R$ and   
\[ \sum_{j=1}^N a_{ij} e^{\i \theta_j}=\lambda_i e^{\i \theta_i} .\]
Then $\bm{\theta}_0=(\theta_1, \theta_2, \ldots, \theta_N)$ is an equilibrium point of the phase-lag KM -- Eq. \eqref{eq:phase_lag}

   \[\frac{d\theta_{i}}{dt} = \epsilon \sum\limits_{j=1}^{N} a_{ij} \sin{(\theta_{j} - \theta_{i} - \phi_i)}, \forall 1 \leq i \leq N \] 
\end{prop}

\begin{rem} \label{rem:relaxed_condition}
Proposition \ref{prop:twisted_stated_phase_lag_KM} is a special case of Proposition  \ref{prop:twisted_stated_phase_lag_KM_gen}, i.e the case when $\lambda_i=\lambda$ for all $1 \leq i \leq N$.
\end{rem} 

We then use computational analyses to further illustrate the equilibrium points for phase-lag Kuramoto oscillators. Figure \ref{fig:eq_points_phase_lag} depicts a graphic representation of this kind of solution for coupled Kuramoto oscillators. Here, the network is given with the same configuration than in Fig. \ref{fig:eq_point_circulant_phases}: $N = 50$ oscillators coupled in a ring network with $k = 10$, the coupling strength is given by $\epsilon = 1.0$, and the natural frequency is given by $\omega = 20 \pi$. Furthermore, the equilibrium points represented here are given by the $2^{\mathrm{nd}}$ eigenvector of the matrix $\bm{K}$, which has information about the network topology, as well as, the phase-lag factor $\phi$. Here, we consider $\phi = 1.00$ (a), and $\phi = \pi/2$, where one can observe that the role of this factor is given by a rotation in the solution in the circle.
\begin{figure}[htb]
    \centering
    \includegraphics[width=0.75\textwidth]{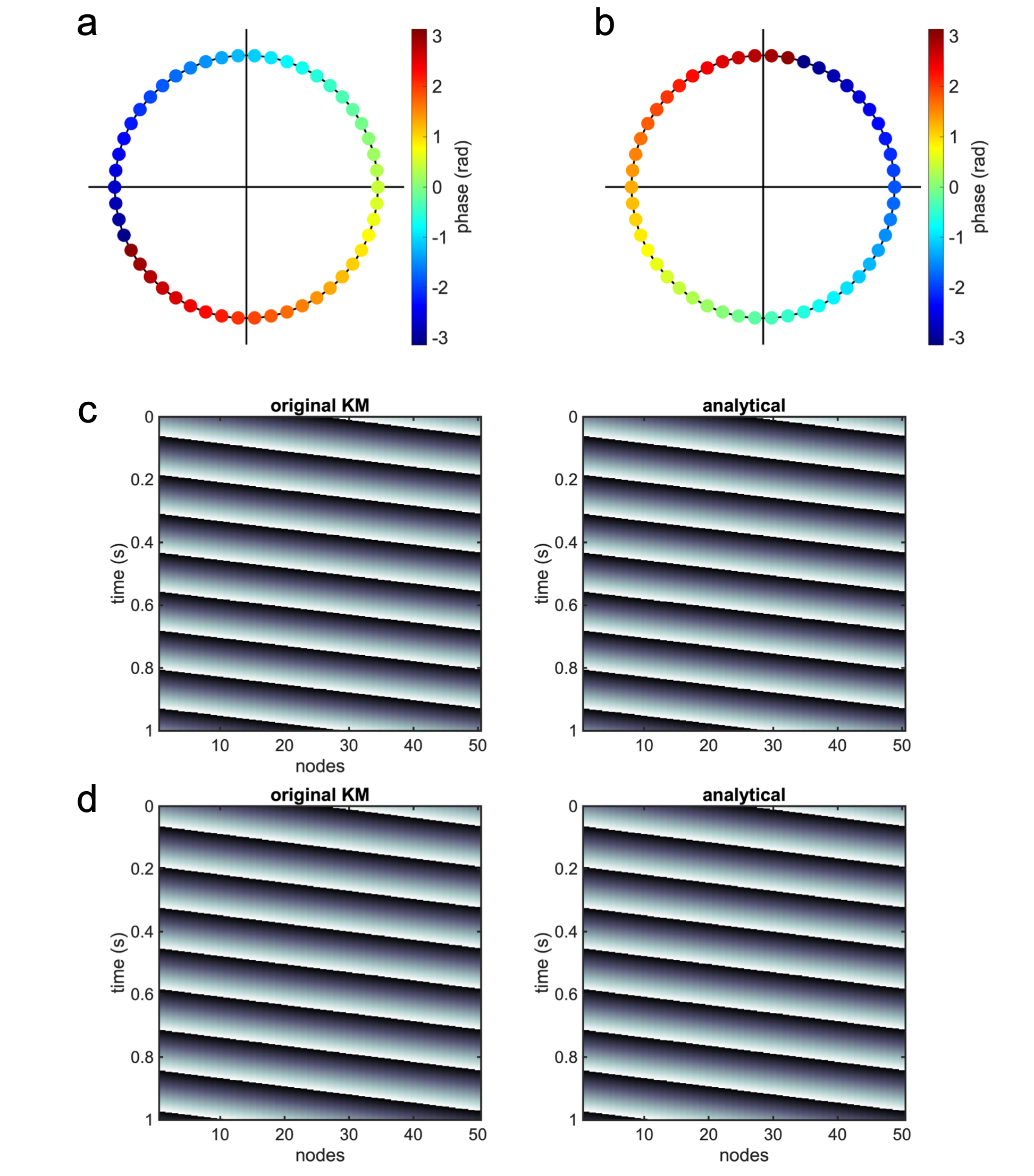}
    \caption{Graphic representation of equilibrium points for phase-lag Kuramoto oscillators. Here, a network with $N = 50$ is considered, where the adjacency matrix is given by ring network, where each node has $k = 10$ connections in both directions. Moreover, the solutions are given by the $2^{\mathrm{nd}}$ eigenvector of the matrix $\bm{A}$ with $\phi = 1.00$ (a), and $\phi = \pi/2$. The effect of the phase-lag is observed in the rotation of the solution. Spatiotemporal dynamics for this system (phase-lag Kuramoto oscillators) is represented in (c) and (d), where the initial conditions follow the equilibrium points represented in (a), and (b), respectively. Both the original Kuramoto model and complex-valued one (analytical) depict the same behavior, where the dynamics remains the same as time evolves, and the equilibrium points can be understood as traveling waves.}
    \label{fig:eq_points_phase_lag}
\end{figure}

In addition to that, Figs. \ref{fig:eq_points_phase_lag}c and \ref{fig:eq_points_phase_lag}d illustrate the spatiotemporal dynamics of the networks considering the initial conditions given by the solutions represented in panels (a) and (b). Here, the wave pattern is observed for the whole analysis, in both the original KM and the complex-valued model (analytical), which corroborates the findings about these equilibrium points. The only difference between panels (c) and (d) is the rotation in the solution observed in panels (a) and (b), therefore leading to a shifting in the oscillators' phases.

\begin{expl}\label{exp:example_1}
Let us consider a network with $4$ nodes and their topological connection is given by the following adjacency matrix 
\[\bm{A}=\begin{pmatrix}
0 &0 &1 &1 \\
1 &0 &0 & 1 \\
1 &1 &0 &0 \\
0 &1 &1 &0
\end{pmatrix} .\]

This is a circulant matrix which is not symmetric. By the CDT, we can see that $\bm{v}=(1, \i, -1, -\i)^{T}$ is an eigenvector of $\bm{A}$ associated with the eigenvalue $\lambda=-(1+\i)$. Note that we have $\bm{v}=e^{\i \bm{\theta}_0}$ with $\bm{\theta}_0=(0, \frac{\pi}{2}, \pi, -\frac{\pi}{2})$. Additionally, we have 
\[ \lambda=-(1+\i)=\sqrt{2}e^{\i (\frac{-3 \pi}{4})}.\]
Let $\phi=\frac{\pi}{4}$ then $\lambda e^{-\i \phi}=-\sqrt{2} \in \R$. By Propositions \ref{prop:twisted_stated_phase_lag_KM} and \ref{prop:twisted_stated_phase_lag} we know that $\bm{\theta}_0$ is an equilibrium point of the KM associated with $\bm{A}$ with phase-lag $\phi=\frac{\pi}{4}$. Concretely, the KM is described by the following system of differential equations 
\[\begin{cases}
\frac{d\theta_1}{dt}= \sin(\theta_3-\theta_1-\frac{\pi}{4})+\sin(\theta_4-\theta_1-\frac{\pi}{4}) \\
\frac{d\theta_2}{dt}= \sin(\theta_1-\theta_2-\frac{\pi}{4})+\sin(\theta_4-\theta_2-\frac{\pi}{4}) \\
\frac{d\theta_3}{dt}= \sin(\theta_1-\theta_3-\frac{\pi}{4})+\sin(\theta_2-\theta_3-\frac{\pi}{4}) \\
\frac{d\theta_4}{dt}= \sin(\theta_2-\theta_4-\frac{\pi}{4})+\sin(\theta_3-\theta_4-\frac{\pi}{4}). \\
\end{cases} \]

Figure \ref{fig:example_1_phase_lag} portrays a graphic representation of the equilibrium point for this example (a). Moreover, panel (b) brings the spatiotemporal dynamics for this network when the initial condition is given by the equilibrium point represented in (a). In this case, the dynamics is given by a traveling wave and corroborates the solution being an equilibrium point.
\begin{figure}[htb]
    \centering
    \includegraphics[width=0.85\textwidth]{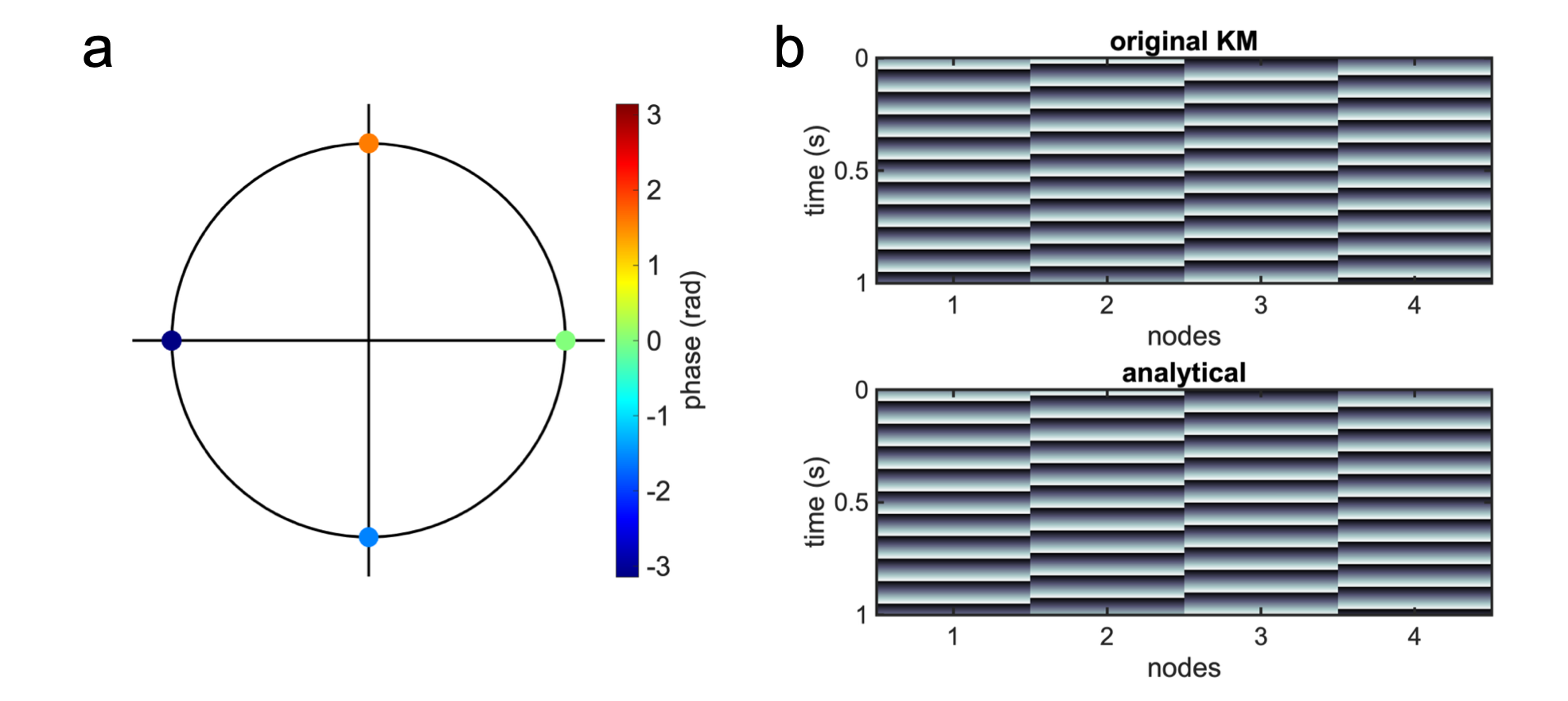}
    \caption{A network following the description of Example \ref{exp:example_1} is analyzed, where panel (a) shows an equilibrium point for this system given by $\theta_{0} = (0, \pi/2, \pi, -\pi/2)$. Panel (b) shows the dynamics of this network when the equilibrium point represented in (a) is used as initial condition.}
    \label{fig:example_1_phase_lag}
\end{figure}
\end{expl}

\section{Equilibrium points for generalized circulant networks}\label{sec:eq_points_generalized_circulant}

In this section, we discuss equilibrium points for Kuramoto networks, where the connection architecture is given by circulant matrices associated with a group $G$. Here, we use the results obtained in the Secs. \ref{sec:eq_points_circulant} and \ref{sec:eq_points_phase_lag} to explore the features of the equilibrium points. First, we recall their definitions (see \cite{kanemitsu2013matrices} for a more thorough discussion.)

\begin{definition}
Let $G$ be a finite group. A matrix $\bm{C}$ is called $G$-circulant if it has the following form 
\[ \bm{C}=(c_{\tau^{-1} \sigma})_{\tau, \sigma \in G} ,\]
where $c_{g} \in \mathbb{C}$ for $g \in G$.
\end{definition} 

\begin{rem}
When $G=\Z/n$, we recover the notion of ``circulant matrices" discussed in Section \ref{sec:eq_points_circulant}.
\end{rem} 

The following proposition is proved implicitly in \cite[Section 1.2]{kanemitsu2013matrices}.
\begin{prop} \label{prop:eigen_circulant}
Let $\chi: G \to \mathbb{C}^{\times}$ be a $1$-dimensional representation of $G$. Then $\bm{v_{\chi}}=(\chi(\sigma))_{\sigma \in G}^{T}$ is an eigenvector of any $G$-circulant matrix $\bm{C}$. The corresponding eigenvalue is 
\[ Y_{\chi}=\sum_{\sigma \in G} c_{\sigma} \chi(\sigma) \]
\end{prop} 

\begin{proof}
For all $\sigma \in G$, the $\tau$-component of the column vector $\bm{C} \bm{v_{\chi}}$ is given by
\[ (\bm{C} \bm{v_{\chi}})_{\tau}=\sum_{\sigma \in G} c_{\tau^{-1} \sigma} \chi(\sigma) .\] 
Letting $g=\tau^{-1} \sigma$, we have $\sigma=\tau g$. Hence
\begin{align*}
(\bm{C} \bm{v_{\chi}})_{\tau} &=\sum_{g \in G} c_{g} \chi(\tau g)\\
                 &=\sum_{g \in G} c_{g} \chi(g) \chi(\tau)\\
                 &=\chi(\tau) \sum_{g \in G} c_{g} \chi(g)\\
                 &=Y_{\chi} \chi(\tau)=Y_{\chi}(\bm{v}_{\chi})_{\tau}.
\end{align*}
Here we use the fact that $\chi(\tau g)=\chi(\tau) \chi(g)$ as $\chi$ is a $1$-dimensional representation of $G$. Since this true for all $\tau \in G$, we conclude that $\bm{v_{\chi}}$ is an eigenvector of $\bm{C}$ associated with the eigenvalue $Y_{\chi}$.
\end{proof}

Let $n=|G|$. Then by Lagrange's theorem, $\sigma^n=1$ for all $\sigma \in G$. Therefore 
\[ \chi(\sigma)^n=\chi(\sigma^n)=\chi(1)=1 .\] 
We conclude that $\chi(\sigma)$ must be an $n$-th root of unity. Therefore, we can define 
\[ \arg(\bm{v_{\chi}})=(\arg(\chi(\sigma))_{\sigma \in G}  \in \frac{2 \pi}{n}{\Z}^n. \]

By Proposition \ref{prop:twisted_stated_phase_lag_KM}, and Proposition \ref{prop:twisted_stated_phase_lag}, we have the following.

\begin{prop} \label{prop:generalized_equilibrium}
Suppose the topological connection of oscillators is given by a $G$-circulant matrix $\bm{A}=(a_{ij})$. Let $\chi, \bm{v_{\chi}}, Y_{\chi}, \arg(\bm{v_{\chi}})$ be as above and $\phi=\arg(Y_{\chi})$. Then 
\begin{enumerate}
\item $\arg(\bm{v_{\chi}})$ is an equilibrium point of the original phase-lag KM 

\[    \frac{d\theta_{i}}{dt} = \epsilon \sum\limits_{j=1}^{N} a_{ij} \sin{(\theta_{j} - \theta_{i} - \phi)} .\]
\item $\arg(\bm{v_{\chi}})$ is an equilibrium point of the associated phase-lag complex-valued model

\[ \frac{d \theta_i}{dt}=\epsilon \sum_{j=1}^N a_{ij} \left(\sin(\theta_j-\theta_i-\phi)-\i \cos(\theta_j-\theta_i-\phi) \right). \] 
\end{enumerate}
\end{prop} 

\section{Equilibrium points for multilayer networks}\label{sec:eq_points_non_circulant}

In this section, we construct some examples of non-circulant network that have interesting  equilibrium points. We do so by applying a recent result on the join of two circulant graphs \cite{djoan2022joins}. First, we introduce the following convention. In the following, the operator $\conc$ denotes vector concatenation:
\[
(x_1,\dots,x_m)^T \conc (y_1,\dots,y_n)^T = (x_1,\dots,x_m, y_1,\dots,y_n)^T
\]
We also denote by $\omega_n= e^{\frac{ 2 \pi \i}{n}}$, a fixed primitive $n$-root of unity in $\mathbb{C}.$

We start with the following observation (see also \cite[Section 2]{djoan2022joins}). 

\begin{prop} \label{prop:block_circulant}
Let $\bm{A}$ be an $(k_1+k_2) \times (k_1+k_2)$ matrix of the form 
\[ \bm{A}= \begin{pmatrix}
\bm{C} & [\alpha]_{k_1, k_2} \\ 
[\beta]_{k_2, k_1} & \bm{D}
\end{pmatrix}. \]

Here $\bm{C}=circ(c_0,\dots,c_{k_1})$ and $\bm{D}=circ(d_0,\dots,d_{k_2})$ are circulant matrices of respective sizes $k_1 \times k_1$ and $k_2 \times k_2$, $[\alpha]_{k_1, k_2}$ is a $k_1 \times k_2$ matrix whose entries are equal to $\alpha\in\R$, and $[\beta]_{k_2, k_1}$ is a $k_2 \times k_1$ matrix whose entries are equal to $\beta\in\R$. For $1 \leq j \leq k_1-1$ let 
\[ \bm{w}_j= (1, \omega_{k_1}^j, \omega_{k_1}^{2j}, \ldots, \omega_{k_1}^{(k_1-1)j}, 0 \ldots, 0 )^T = \bm{v}_{j,k_1} \conc \underbrace{(0, 0, \ldots, 0)^T}_{\text{$k_2$ zeros}} \] 
with  
\[ \bm{v}_{j,k_1}=(1, \omega_{k_1}^j, \omega_{k_1}^{2j}, \ldots, \omega_{k_1}^{(k_1-1)j})^T. \] 

For $1 \leq j \leq k_2-1$, let 
\[ \bm{z}_j= (0 \ldots, 0  ,1, \omega_{k_2}^j, \omega_{k_2}^{2j}, \ldots, \omega_{k_2}^{(k_2-1)j})^T= \underbrace{(0, 0, \ldots, 0)^T}_{\text{$k_1$ zeros}} * \bm{v}_{j,k_2}  , \]
with 
\[ \bm{v}_{j,k_2}=  (1, \omega_{k_2}^j, \omega_{k_2}^{2j}, \ldots, \omega_{k_2}^{(k_2-1)j})^T. \] 
We have the following 
\begin{enumerate}
\item $\bm{w}_j$ is an eigenvector of $\bm{A}$ associated with the eigenvalue 
\[ \lambda _{j}^{\bm{C}}=c_{0}+c_{k_1-1}\omega _{k_1}^{j}+c_{k_1-2}\omega_{k_1} ^{2j}+\dots +c_{1}\omega_{k_1} ^{(k_1-1)j} \] 

\item Similarly, 
 $\bm{z}_j$ is an eigenvector associated with the eigenvalue 
\[ \lambda _{j}^{\bm{D}}=c_{0}+d_{k_2-1}\omega _{k_2}^{j}+d_{k_2-2}\omega_{k_2} ^{2j}+\dots +d_{1}\omega_{k_2} ^{(k_2-1)j} \] 
\end{enumerate}
\end{prop} 

\begin{proof}
By the Circulant Diagonalization Theorem, $\bm{v}_{j,k_1}$ is an eigenvector of $\bm{C}$ with respect to the eigenvalue $\lambda_{C}^{j}$. 
By definition, we have 
\[ \bm{A} \bm{w}_j= \bm{C} \bm{v}_{j,k_1} \conc  \underbrace{(t_{j} , t_{j}, \ldots, t_{j})^T}_{\text{$n-k$ terms}}=\lambda_j^{\bm{C}} \bm{v}_{j,k_1}   \conc \underbrace{(t_{j} , t_{j}, \ldots, t_{j})^T}_{\text{$n-k$ terms}}\] 
Here 
\[ t_j =\beta \sum_{i=0}^{k-1} \omega_{k}^{i j} .\]
By the assumption $1 \leq j \leq k_1-1$, we can see that $t_j=0$. Therefore, we conclude that 
\[ \bm{A} \bm{w}_j=\lambda_j^{\bm{C}} \bm{w}_j .\] 
This proves the first statement. The second statement can be proved by the same argument.
\end{proof}

Let us consider a special case when $\bm{C} = \bm{D}$ and $\bm{C}$ is a symmetric circulant matrix (so in particular, $k_1=k_2=k$). In this case, $\bm{w}_j$ and $\bm{z}_j$ are both eigenvectors with respect to the same eigenvalue $\lambda_{j}^{\bm{C}}$. Note that by assumption $\bm{C}$ is symmetric, so $\lambda_{\bm{C}}^{j} \in \R$.  Furthermore, for any $\phi\in[0,2\pi)$, $\bm{w}_j + e^{\i\phi} \bm{z}_j$ is an eigenvector with respect to $\lambda_{\bm{C}}^{j}$. We observe that
\begin{align*} 
\bm{w}_j + e^{\i\phi} \bm{z}_j &=\bm{v}_{j,k} * e^{\i\phi}\bm{v}_{j,k} \\  &= (1, \omega_{k}^j, \omega_{k}^{2j}, \ldots, \omega_{k}^{(k-1)j}, e^{\i\phi}, e^{\i\phi}\omega_{k}^j, e^{\i\phi} \omega_{k}^{2j}, \ldots, e^{\i\phi}\omega_{k}^{(k-1)j})^T \\ 
&=(e^{\i \theta_1}, e^{\i \theta_2}, \ldots, e^{\i \theta_{k}}, e^{\i \theta_{k+1}}, e^{\i \theta_{k+2}}, \ldots, e^{\i \theta_{2k}})^T,
\end{align*}
where 
\begin{align*}
\bm{\theta}_{0}^{(j)} &=(\theta_1, \theta_2, \ldots, \theta_{k}, \theta_{k+1}, \theta_{k+2}, \ldots, \theta_{2k})^T\\ 
&=\left(0, \frac{2 \pi j}{k}, \ldots, \frac{2 \pi(k-1)j}{k}, \phi, \frac{2 \pi j}{k}+\phi, \ldots, \frac{2 \pi(k-1)j}{k}+\phi \right)^T. 
\end{align*}
By this argument, Proposition \ref{prop:twisted_state_KM}, and Proposition \ref{prop:twisted_state}, we have the following.

\begin{prop} \label{prop:example_non_circulant}

Let $\bm{C}$ be a symmetric $k\times k$ circulant matrix, and $\alpha, \beta$ be two arbitrary real numbers. Consider the following matrix 
\[ \bm{A}= \begin{pmatrix}
\bm{C} & [\alpha]_{k, k} \\ 
[\beta]_{k,k} & \bm{C}
\end{pmatrix}. \]

For each $1 \leq j \leq k-1$, and for any $\phi\in[0,2\pi)$, let 
\[ \bm{\theta}_{0}^{(j,\phi)}= \left(0, \frac{2 \pi j}{k}, \ldots, \frac{2 \pi(k-1)j}{k}, \phi, \frac{2 \pi j}{k}+\phi, \ldots, \frac{2 \pi(k-1)j}{k}+\phi \right)^T .\] 
\begin{enumerate}
\item $\bm{\theta}_{0}^{(j,\phi)}$ is an equilibrium point of the complex-valued model. 
\item $\bm{\theta}_{0}^{(j, \phi)}$ is also an equilibrium point of the original KM. 
\end{enumerate}
\end{prop}

\begin{rem}
Proposition \ref{prop:example_non_circulant} can be generalized to the case where we join $d$ identical circulant networks. We refer interested readers to \cite[Section 5]{djoan2022joins} for further details. 
\end{rem} 

Based on the results depicted in this section, we perform computational analyses for the example demonstrated in proposition \ref{prop:example_non_circulant}. Then, Fig. \ref{fig:eq_points_non_circulant}(a) shows the graphic representation of an equilibrium point for a network with $N = 50$ nodes following $\bm{\theta}_0$ with $p = 1$. Furthermore, the phases of the first $N/2$ nodes is represented in bigger circle, while the phases of the last $N/2$ nodes is given by the smaller ones. Here, the network is described by a non-circulant graph represented by matrix $\bm{A}$, where $\bm{C}$ follows a ring network with $k = 5$, $\alpha = 0.25$ and $\beta = 0.75$. A graphic representation of this non-circulant matrix is given by Fig. \ref{fig:eq_points_non_circulant}b. 
\begin{figure}[htb]
    \centering
    \includegraphics[width=0.75\textwidth]{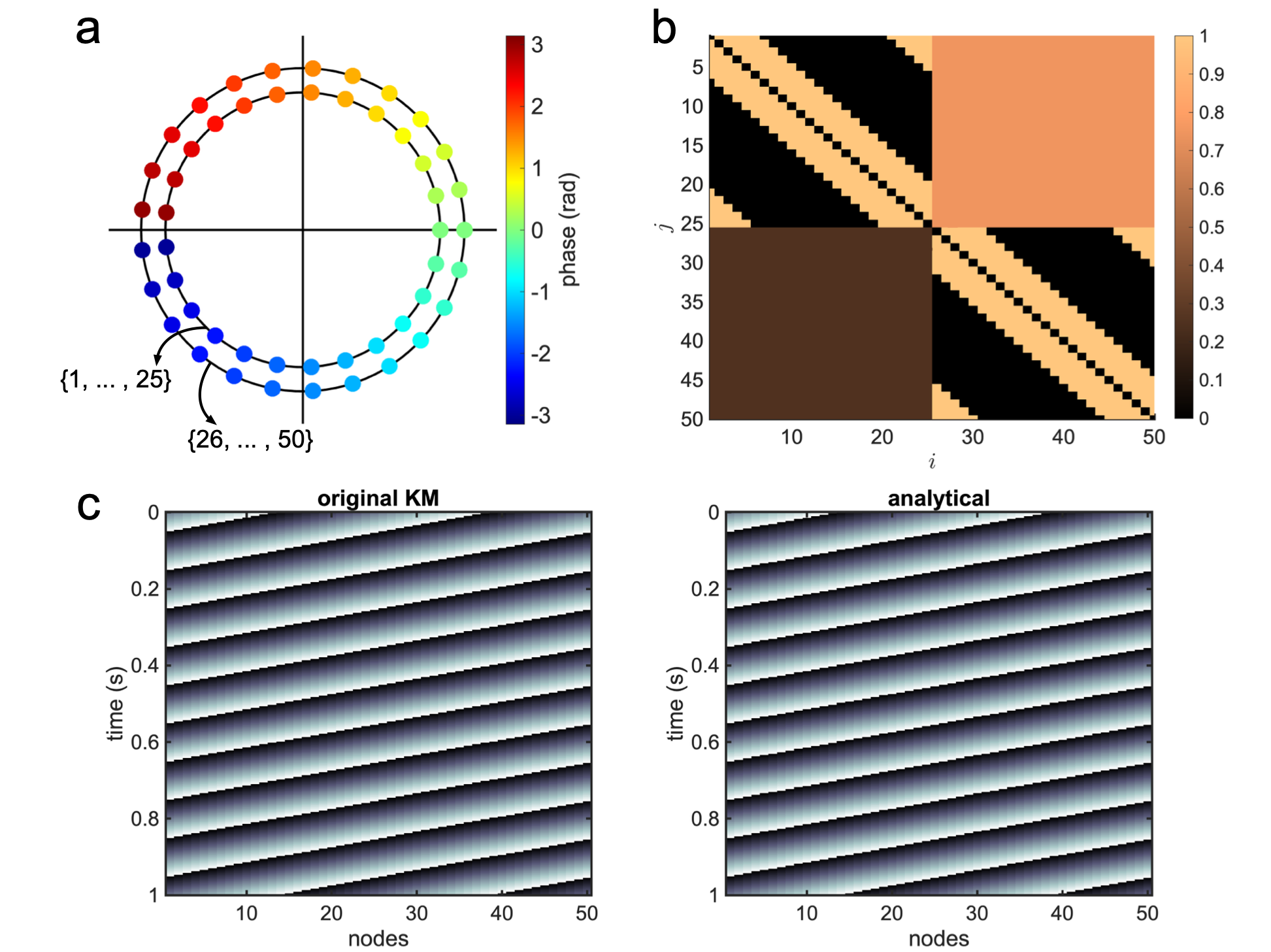}
    \caption{An equilibrium point for a non-circulant network with $N = 50$ is graphically represented in (a), where the phase of each node is depicted in color-code. Here, the graph is given by the matrix represented in (b), which is composed by circulant matrices. The spatiotemporal dynamics for this system when the equilibrium point (a) is used as initial condition is represented in panel (c) for both the original KM and the complex-valued model (analytical). The matrix and the equilibrium point follow the example in proposition \ref{prop:example_non_circulant}.}
    \label{fig:eq_points_non_circulant}
\end{figure}

Moreover, Fig. \ref{fig:eq_points_non_circulant}c represents the spatiotemporal patterns for both the original KM and the complex-valued model (analytical). Here, the analyses use the equilibrium point represented in Fig. \ref{fig:eq_points_non_circulant}a as initial condition, so we can observe the wave pattern as time evolves.

\section{Equilibrium points for random networks}\label{sec:eq_points_random_networks}

Our analytical approach allows us to ``design'' a twisted state (equilibrium point) on an undirected random network, given by a Erd\H{o}s-R\'{e}yni graph with $N = 100$, and $p = 0.25$ (Fig. \ref{fig:random_network_twisted_state}a). We first evaluate the eigenspectrum of this random matrix numerically, and then decompose the matrix using its eigenvectors and eigenvalues ($A = VDV'$). Eigenvalues were arranged in ascending order by their real part. We then modified the $2^{\mathrm{nd}}$ and $3^{\mathrm{rd}}$ last eigenvectors by applying the sine and cosine functions, respectively, to the phase given by Eq.~(\ref{eq:twisted_state_eq}). Finally, we set the eigenvalues associated with these eigenvectors to be equal and scaled appropriately. Then, using $V D V'$, we create a modified matrix $\bm{A}'$ (Fig. \ref{fig:random_network_twisted_state}b), which is used as the adjacency matrix for the simulations with the original KM. With this modification $\cos(\bm{\theta}_0)$ and $\sin(\bm{\theta}_0)$ are eigenvectors of the new weighted adjacency matrix associated to the same real eigenvalue $\lambda$. Consequently $e^{\i \bm{\theta}_0}=\cos(\bm{\theta}_0)+\i \sin(\bm{\theta}_0)$ is an eigenvector of the new matrix $\bm{A}'$ associated with the eigenvalue $\lambda$. By Proposition \ref{prop:twisted_state_KM}, we know that $\bm{\theta}_0$ is an equilibrium point of the new system associated with $\bm{A}'$. Finally, it is important to note that these results were for these two eigenvalues scaled by approximately one order of magnitude, which demonstrates that this result does not depend on the modified eigenvalue-eigenvector pairs dominating the resulting system dynamics.
\begin{figure}[htb]
    \centering
    \includegraphics[width=0.75\textwidth]{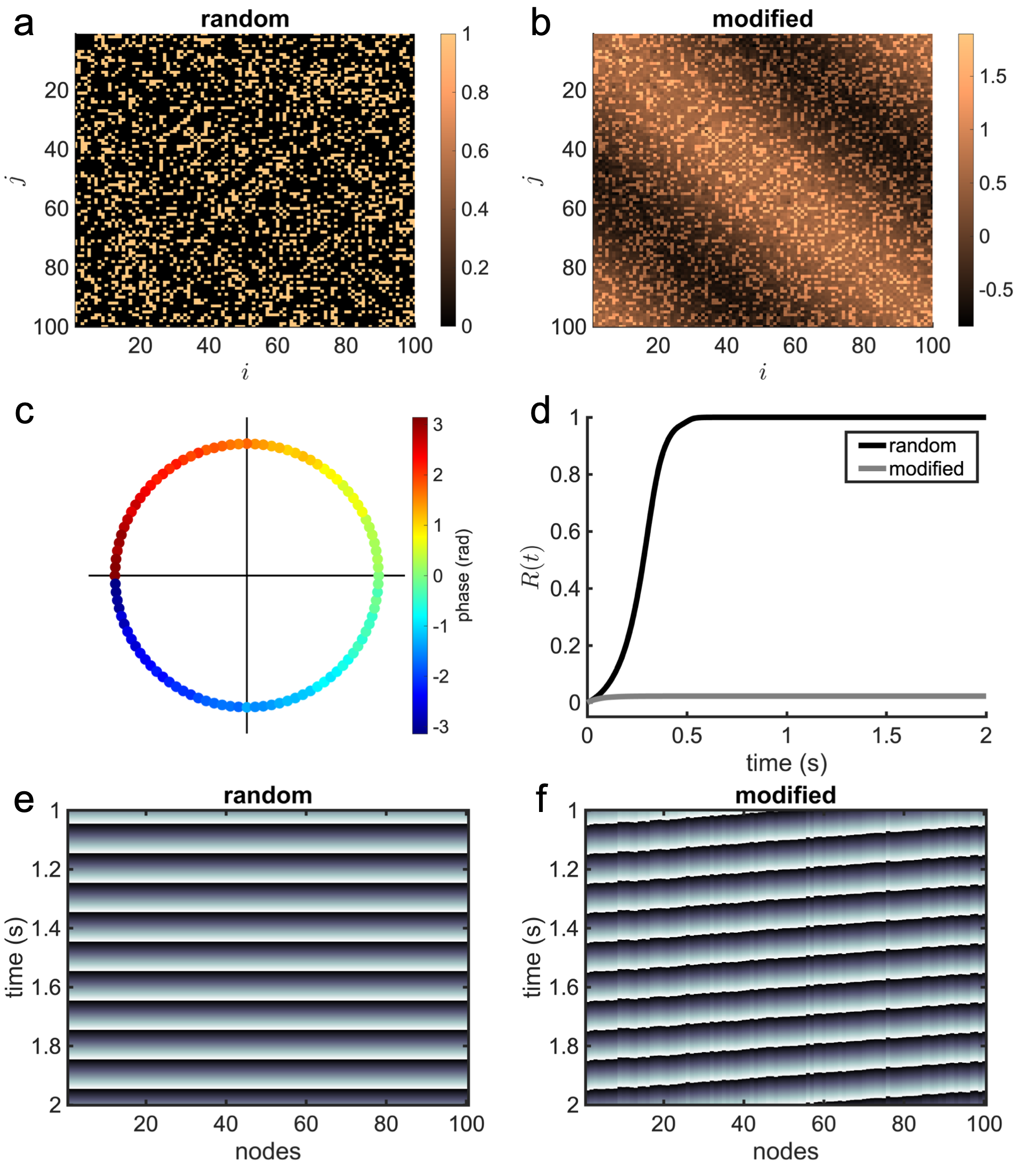}
    \caption{We modify a random (Erd\H{o}s-Renyi) matrix in order to produce an equilibrium point given by a twisted state. Here, the original matrix (a) is modified (b) due to changes in some of the eigenvectors. The equilibrium point (c) is then used as initial condition to the simulation. The Kuramoto order parameter ($R(t)$) as a function of time shows that for the random matrix, the system reaches a phase synchronized state, while for the modified matrix, the system stays in a wave (``twisted") state. The spatiotemporal patterns corroborate these features (e, and f).}
    \label{fig:random_network_twisted_state}
\end{figure}

The equilibrium point for these systems is then represented in Fig. \ref{fig:random_network_twisted_state}c, which is given by Eq. (\ref{eq:twisted_state_eq}). Using this phase configuration as initial condition for the simulation leads the systems to different states: in the case of the random matrix, the system reaches a phase synchronized state ($R = 1$); in the case of the modified matrix, the system stays in a twisted state, which is a phase-locked but not phase synchronized state ($R = 0$) -- see Fig. \ref{fig:random_network_twisted_state}d. These features can be observed in the spatiotemporal dynamics of these networks, which are depicted in Figs. \ref{fig:random_network_twisted_state}e and f, respectively.

\section{Discussion and conclusions}\label{sec:discussions_conclusions}

In this paper, we have analyzed equilibria in Kuramoto systems. To do so, we have used a complex-valued version of the Kuramoto model \cite{muller2021algebraic}, which allows us to further investigate equilibria in the original nonlinear Kuramoto model. In this context, we have shown that some of the eigenvectors of the adjacency matrix are equilibrium points for the original Kuramoto model and also for the complex-valued model. These results thus indicate there is a strong correspondence between the original Kuramoto model, which is given by nonlinear differential equations, and the complex-valued one, which admits an exact analytical solution for individual realizations of the system and on finite graphs \cite{budzinski2022geometry}. Using this approach, we are able to extend the analysis of equilibria in Kuramoto networks to new conditions not previously considered.

Based on the general result presented in Sec. \ref{sec:eq_points_kuramoto}, where we have shown that we can find equilibrium points by using the eigenvectors of the adjacency matrix for Kuramoto systems, we first considered equilibria for the well-known case of circulant networks. We then moved on to study the case of global, all-to-all connections (complete graph on $N$ nodes), which is the case first studied by Kuramoto, where we have completely characterized the equilibrium points.

We then have analyzed equilibrium points for phase-lag oscillators, where the imaginary part of the eigenvalues becomes important. This result allows us to understand equilibria in Kuramoto system with an interplay between attractive and repulsive coupling, due to the phase-lag parameter. We also have analyzed the case of generalized circulant graphs, where the eigenvectors and eigenvalues can be also imaginary, therefore affecting the equilibria of this kind of system.

Furthermore, we have studied equilibria in the case of multilayer networks, where the system can be analyzed as the join of circulant networks. Based on \cite{djoan2022joins}, we have found equilibrium points for this kind of system, which opens the possibility of application in several systems for the study of spreading dynamics, neuroscience, synchronization, technical system, and others \cite{boccaletti2014structure, bassett2017network, de2016physics, kivela2014multilayer}.

Finally, we have used the ideas developed in this paper to ``design'' an equilibrium point in a random network. In this case, we have shown a procedure to change the adjacency matrix in order to create new equilibria in the system. This shows the utility of our analytical approach and extends the study of equilibria to a class of networks beyond the circulant graphs considered previously. Based on this result, future work can extend these ideas to the controlling Kuramoto systems, demonstrating the utility of analyzing the original, nonlinear Kuramoto model through the lens of the complex-valued approach \cite{muller2021algebraic, budzinski2022geometry}.

Throughout this paper, we have shown a novel approach to investigate equilibria in Kuramoto networks. Based on our complex-valued model for Kuramoto oscillators, we can now study, analytically, network of Kuramoto oscillators under new and more varied conditions. Our study reveals a strong correspondence with the original, nonlinear, Kuramoto model, where we can now track, analytically, individual realizations of oscillator networks and push further investigations of the rich dynamics that this kind of system offers.

\section*{Appendix - Computational analyses}\label{sec:numerical_simulations}

The solution for the original Kuramoto model (KM) is given by the numerical integration of Eq. (\ref{eq:KM}), where we use Euler's method with time step of $10^{-4}$. On the other hand, the analytical solution is obtained through the evaluating of Eq. (\ref{eq:solution_analytical_model}) using a mathematical tool called Expokit \cite{sidje1998expokit}, which is designed to solve exponential matrix equations. Here, we used a windowed approach to the propagation of the solution, where the final solution for each window as the initial condition for the subsequent one. A detailed explanation on this point can be found in \cite{budzinski2022geometry}.

\section*{Acknowledgments}
This work was supported by BrainsCAN at Western University through the Canada First Research Excellence Fund (CFREF), the NSF through a NeuroNex award (\#2015276), the Natural Sciences and Engineering Research Council of Canada (NSERC) grant R0370A01, and by the Western Academy for Advanced Research. J.M.~gratefully acknowledges the Western University Faculty of Science Distinguished Professorship in 2020-2021. R.C.B gratefully acknowledges the Western Institute for Neuroscience Clinical Research Postdoctoral Fellowship.


\begin{thebibliography}{37}

\bibitem{Kuramoto2012chemical}
Y.~Kuramoto.
\newblock {\em Chemical oscillations, waves, and turbulence}.
\newblock Springer Science \& Business Media, 2012.

\bibitem{acebron2005kuramoto}
J.~A. Acebr{\'o}n, L.~L. Bonilla, C.~J.~P. Vicente, F.~Ritort, and R.~Spigler.
\newblock The kuramoto model: A simple paradigm for synchronization phenomena.
\newblock {\em Reviews of Modern Physics}, 77(1):137, 2005.

\bibitem{abrams2004chimera}
D.~M. Abrams and S.~H. Strogatz.
\newblock Chimera states for coupled oscillators.
\newblock {\em Physical Review Letters}, 93(17):174102, 2004.

\bibitem{arenas2008synchronization}
A.~Arenas, A.~Diaz-Guilera, J.~Kurths, Y.~Moreno, and C.~Zhou.
\newblock Synchronization in complex networks.
\newblock {\em Physics Reports}, 469(3):93--153, 2008.

\bibitem{boccaletti_2002}
S.~Boccaletti, J.~Kurths, G.~Osipov, D.~L. Valladares, and C.~S. Zhou.
\newblock The synchronization of chaotic systems.
\newblock {\em Physics Reports}, 366(1-2):1--101, 2002.

\bibitem{parastesh2020chimeras}
F.~Parastesh, S.~Jafari, H.~Azarnoush, Z.~Shahriari, Z.~Wang, S.~Boccaletti,
  and M.~Perc.
\newblock Chimeras.
\newblock {\em Physics Reports}, 2020.

\bibitem{rodrigues2016kuramoto}
F.~A. Rodrigues, T.~K. D.~M. Peron, P.~Ji, and J.~Kurths.
\newblock The kuramoto model in complex networks.
\newblock {\em Physics Reports}, 610:1--98, 2016.

\bibitem{strogatz2000kuramoto}
S.~H. Strogatz.
\newblock From kuramoto to crawford: exploring the onset of synchronization in
  populations of coupled oscillators.
\newblock {\em Physica D: Nonlinear Phenomena}, 143(1-4):1--20, 2000.

\bibitem{xu2018origin}
C.~Xu, S.~Boccaletti, S.~Guan, and Z.~Zheng.
\newblock Origin of bellerophon states in globally coupled phase oscillators.
\newblock {\em Physical Review E}, 98(5):050202, 2018.

\bibitem{Strogatz2001}
S.~H. Strogatz.
\newblock Exploring complex networks.
\newblock {\em Nature}, 410(6825):268--276, 2001.

\bibitem{bick2020understanding}
C.~Bick, M.~Goodfellow, C.~R. Laing, and E.~A. Martens.
\newblock Understanding the dynamics of biological and neural oscillator
  networks through exact mean-field reductions: a review.
\newblock {\em The Journal of Mathematical Neuroscience}, 10:1--43, 2020.

\bibitem{chen2018counting}
T.~Chen, R.~Davis, and D.~Mehta.
\newblock Counting equilibria of the kuramoto model using birationally
  invariant intersection index.
\newblock {\em SIAM Journal on Applied Algebra and Geometry}, 2(4):489--507,
  2018.

\bibitem{xin2016analytical}
X.~Xin, T.~Kikkawa, and Y.~Liu.
\newblock Analytical solutions of equilibrium points of the standard kuramoto
  model: 3 and 4 oscillators.
\newblock In {\em 2016 American Control Conference (ACC)}, pages 2447--2452.
  IEEE, 2016.

\bibitem{chen2019three}
T.~Chen, J.~Mare{\v{c}}ek, D.~Mehta, and M.~Niemerg.
\newblock Three formulations of the kuramoto model as a system of polynomial
  equations.
\newblock In {\em 2019 57th Annual Allerton Conference on Communication,
  Control, and Computing (Allerton)}, pages 810--815. IEEE, 2019.

\bibitem{mehta2015algebraic}
D.~Mehta, N.~S. Daleo, F.~D{\"o}rfler, and J.n~D. Hauenstein.
\newblock Algebraic geometrization of the kuramoto model: Equilibria and
  stability analysis.
\newblock {\em Chaos: An Interdisciplinary Journal of Nonlinear Science},
  25(5):053103, 2015.

\bibitem{coss2018locating}
O.~Coss, J.~D. Hauenstein, H.~Hong, and D.~K. Molzahn.
\newblock Locating and counting equilibria of the kuramoto model with rank-one
  coupling.
\newblock {\em SIAM Journal on Applied Algebra and Geometry}, 2(1):45--71,
  2018.

\bibitem{pikovsky2015dynamics}
A.~Pikovsky and M.~Rosenblum.
\newblock Dynamics of globally coupled oscillators: Progress and perspectives.
\newblock {\em Chaos: An Interdisciplinary Journal of Nonlinear Science},
  25(9):097616, 2015.

\bibitem{pikovsky2008partially}
A.~Pikovsky and M.~Rosenblum.
\newblock Partially integrable dynamics of hierarchical populations of coupled
  oscillators.
\newblock {\em Physical Review Letters}, 101(26):264103, 2008.

\bibitem{basnarkov2007phase}
L.~Basnarkov and V.~Urumov.
\newblock Phase transitions in the kuramoto model.
\newblock {\em Physical Review E}, 76(5):057201, 2007.

\bibitem{pazo2005thermodynamic}
D.~Paz{\'o}.
\newblock Thermodynamic limit of the first-order phase transition in the
  kuramoto model.
\newblock {\em Physical Review E}, 72(4):046211, 2005.

\bibitem{medvedev2014small}
G.~S. Medvedev.
\newblock Small-world networks of kuramoto oscillators.
\newblock {\em Physica D: Nonlinear Phenomena}, 266:13--22, 2014.

\bibitem{hu2014exact}
X.~Hu, S.~Boccaletti, W.~Huang, X.~Zhang, Z.~Liu, S.~Guan, and C-H Lai.
\newblock Exact solution for first-order synchronization transition in a
  generalized kuramoto model.
\newblock {\em Scientific Reports}, 4(1):1--6, 2014.

\bibitem{li2022mean}
W.~Li and H.~Park.
\newblock Mean field kuramoto models on graphs.
\newblock {\em arXiv preprint arXiv:2203.00142}, 2022.

\bibitem{lu2020synchronization}
J.~Lu and S.~Steinerberger.
\newblock Synchronization of kuramoto oscillators in dense networks.
\newblock {\em Nonlinearity}, 33(11):5905, 2020.

\bibitem{taylor2012there}
R.~Taylor.
\newblock There is no non-zero stable fixed point for dense networks in the
  homogeneous kuramoto model.
\newblock {\em Journal of Physics A: Mathematical and Theoretical},
  45(5):055102, 2012.

\bibitem{townsend2020dense}
A.~Townsend, M.~Stillman, and S.~H. Strogatz.
\newblock Dense networks that do not synchronize and sparse ones that do.
\newblock {\em Chaos: An Interdisciplinary Journal of Nonlinear Science},
  30(8):083142, 2020.

\bibitem{yoneda2021lower}
R.~Yoneda, T.~Tatsukawa, and J.~Teramae.
\newblock The lower bound of the network connectivity guaranteeing in-phase
  synchronization.
\newblock {\em Chaos: An Interdisciplinary Journal of Nonlinear Science},
  31(6):063124, 2021.

\bibitem{muller2021algebraic}
L.~Muller, J.~Min\'a\ifmmode~\check{c}\else \v{c}\fi{}, and T.~T. Nguyen.
\newblock Algebraic approach to the kuramoto model.
\newblock {\em Physical Reiew. E}, 104:L022201, Aug 2021.

\bibitem{budzinski2022geometry}
R.~C. Budzinski, T.~T. Nguyen, J.~{\DJ}o{\`a}n, J.~Min{\'a}{\v{c}}, T.~J.
  Sejnowski, and L.~E. Muller.
\newblock Geometry unites synchrony, chimeras, and waves in nonlinear
  oscillator networks.
\newblock {\em Chaos: An Interdisciplinary Journal of Nonlinear Science},
  32(3):031104, 2022.

\bibitem{Perko2001}
L.~Perko.
\newblock {\em Differential Equations and Dynamical Systems}.
\newblock Springer-Verlag, New York, 2001.

\bibitem{kanemitsu2013matrices}
S.~Kanemitsu and M.~Waldschmidt.
\newblock Matrices of finite abelian groups, finite fourier transform and
  codes.
\newblock {\em Proc. 6th China-Japan Sem. Number Theory, World Sci.
  London-Singapore-New Jersey}, pages 90--106, 2013.

\bibitem{djoan2022joins}
J.~{\DJ}o{\`a}n, J.~Min{\'a}{\v{c}}, L.~Muller, T.~T. Nguyen, and F.~W. Pasini.
\newblock Joins of circulant matrices.
\newblock {\em Linear Algebra and its Applications}, 2022.

\bibitem{boccaletti2014structure}
S.~Boccaletti, G.~Bianconi, R.~Criado, C.~I. Del~Genio, J.~G{\'o}mez-Gardenes,
  M.~Romance, I.~Sendina-Nadal, Z.~Wang, and M.~Zanin.
\newblock The structure and dynamics of multilayer networks.
\newblock {\em Physics Reports}, 544(1):1--122, 2014.

\bibitem{bassett2017network}
D.~S. Bassett and O.~Sporns.
\newblock Network neuroscience.
\newblock {\em Nature neuroscience}, 20(3):353--364, 2017.

\bibitem{de2016physics}
M.~De~Domenico, C.~Granell, M.~A. Porter, and A.~Arenas.
\newblock The physics of spreading processes in multilayer networks.
\newblock {\em Nature Physics}, 12(10):901--906, 2016.

\bibitem{kivela2014multilayer}
M.~Kivel{\"a}, A.~Arenas, M.~Barthelemy, J.~P. Gleeson, Y.~Moreno, and M.~A.
  Porter.
\newblock Multilayer networks.
\newblock {\em Journal of Complex Networks}, 2(3):203--271, 2014.

\bibitem{sidje1998expokit}
R.~B. Sidje.
\newblock Expokit: A software package for computing matrix exponentials.
\newblock {\em ACM Transactions on Mathematical Software (TOMS)},
  24(1):130--156, 1998.

\end{thebibliography}
\end{document}